\documentclass[11pt, oneside]{amsart}
\usepackage[colorlinks=true,allcolors=blue!80!black]{hyperref}
\usepackage{tikz, tikz-cd}
\usetikzlibrary{matrix, snakes, patterns, shadows.blur, backgrounds}

\usepackage{amssymb, amsthm}
\usepackage{enumerate, amsfonts,epsfig,color}
\usepackage{MnSymbol}
\usepackage[margin=1.25in]{geometry}                		
\usepackage{graphicx}
\usepackage{subcaption}
\usepackage{autobreak}
\usepackage{amsmath,amscd}	
\usepackage{mathtools}
\usepackage{tikz-cd}
\usepackage{soul}
\usepackage{faktor}
\usepackage{braids}
\usetikzlibrary{braids}

\usepackage{xypic}

\input xy
\xyoption{all}

\newtheorem {theorem}{Theorem}[section]
\newtheorem {lemma} [theorem] {Lemma}
\newtheorem {proposition} [theorem] {Proposition}
\newtheorem {corollary} [theorem] {Corollary}

\newtheorem {remark}[theorem]{Remark}

\newtheorem{mainthm}{Theorem}

\theoremstyle{definition}
\newtheorem{definition}[theorem]{Definition}

\newtheorem{example}[theorem]{Example}

\def\C {\mathbb C}
\def\R {\mathbb R}

\def\PP {\mathbb P}
\def\Hy{\mathbb{H}}

\def\F {\mathbb F}
\def\Z {\mathbb{Z}}

\def\tB {\widetilde{B}}
\def\tP {\widetilde{P}}

\def\Aut {\mathrm{Aut}}
\def\Out {\mathrm{Out}}

\def\Mod {\mathrm{Mod}}

\DeclareMathOperator{\Isom}{Isom}

\DeclareMathOperator{\SO}{SO}

\DeclareMathOperator{\PSL}{PSL}

\DeclareMathOperator{\Bl}{Bl}
\DeclareMathOperator{\GL}{GL}
\DeclareMathOperator{\SL}{SL}

\DeclareMathOperator{\Diff}{Diff}

\title[]{CAT(0) geometry of complex curve complements and families}

\author{Corey Bregman}
\address{Department of Mathematics \\ Tufts University}
\email{Corey.Bregman@tufts.edu}
\urladdr{https://sites.google.com/view/cbregman}

\author{Anatoly Libgober}
\address{Department of Mathematics, Statistics and Computer Science\\ University of Illinois at Chicago}
\email{libgober@uic.edu}
\urladdr{http://homepages.math.uic.edu/~libgober/} 

 \author{Kejia Zhu}
\address{Department of Mathematics\\ University of California at Riverside}
 \email{kzhumath@gmail.com}
\urladdr{https://sites.google.com/view/kejiazhu} 

\begin{document}
\maketitle

\begin{abstract}

Motivated by the question of whether braid groups are CAT(0), we investigate the CAT(0) behavior of fundamental groups of plane curve complements and certain universal families. If $C$ is the branch locus of a generic projection of a smooth, complete intersection surface to $\PP^2$, we show that $\pi_1(\PP^2\setminus C)$ is CAT(0). In the other direction, we prove that the fundamental group of the universal family associated with the singularities of type $E_6$, $E_7$, and $E_8$ is not CAT(0). 
\end{abstract}

\section{Introduction}

The purpose of the present article is to investigate some properties of complex plane curve complements and their fundamental groups via the lens of geometric group theory.  On the one hand, plane curve singularities bear an intimate connection with the topology of 3-dimensional iterated torus link complements (see, for example, \cite{eisenbud2016three}). On the other hand, work going back to Zariski \cite{zariski1936poincare} and Moishezon \cite{Moishezon81} highlights the close relationship between fundamental groups of many plane curves complements and braid groups. More generally, it turns out that Artin groups, as natural generalizations of braid groups, are also related to the fundamental groups of some specific plane curves (cf.\cite[Theorem 2.6]{artal2018wirtinger}). Since both braid groups, more generally Artin groups and 3-manifold groups (with or without boundaries) are well-studied objects in geometric group theory for the richness of their intrinsic geometry, fundamental groups of plane curve complements form a natural setting to explore commonalities between these two classes.

Fundamental groups of plane curve complement fit into the larger class of \emph{quasi-projective groups}, \emph{i.e.} fundamental groups of smooth, quasi-projective varieties. Both classes are broad, and restrictions on such groups can often be quite subtle. Nevertheless, some of the topological connections described above are reflected even in the larger quasi-projective setting.
For example, Friedl--Suciu \cite{friedl2014kahler} have shown that if a quasi-projective group $G$ is isomorphic to the fundamental group of a compact 3-manifold $N$ with empty or toroidal boundary, then $N$ is a graph manifold. In this sense, $G$ also resembles the fundamental group of an iterated torus link complement, as in the case of a plane curve singularity.  
 

Geometrically, results of Leeb \cite{Leeb1995} imply that graph manifolds with non-empty boundary (and hence iterated torus link complements) admit Riemannian metrics of nonpositive curvature. In particular, their fundamental groups are CAT(0) in the sense that they admit proper, cocompact actions on CAT(0) spaces. CAT(0) geometry, introduced by Gromov \cite{Gromov}, is a natural extension of the concept of a nonpositive curvature to general metric spaces and conjecturally, braid groups are CAT(0) as well. Motivated by the intrinsic connections between 3-dimensional graph manifold groups, braid groups, and CAT(0) geometry, in this paper, we focus on the relation between fundamental groups of plane curve complements and CAT(0) geometry. Our principal goal is to provide both positive and negative results in this vein, as well as raise some questions for further investigation. In the positive direction, we show
\begin{mainthm}\label{thm:mainA}
Let $C\subset \PP^2$ be a plane curve. 
\begin{enumerate}
    \item Fix $O\in\PP^2\setminus C$ and let $L$ be the union of singular fibers from the projection of the blow-up $\Bl_O(\PP^2)\setminus C\rightarrow \PP^1$. If  $\Bl_O(\PP^2)\setminus (C\cup L)$ fibers with finite monodromy, then $\pi_1(\Bl_O(\PP^2)\setminus (C\cup L))$ is CAT(0).
    \item If $C$ is the branch locus of a generic finite map $X\rightarrow \PP^2$, arising from a smooth complete intersection surface $X\subset \PP^N$, then $\pi_1(\PP^2\setminus C)$ is CAT(0).
\end{enumerate}

\end{mainthm}
For background on blow-ups and construction of the fibration described in part (1), see Sections \ref{sec:ZVK} and  \ref{A1} in the Appendix.  The first part of this theorem will be known to the experts, and in this case $\Bl_O(\PP^2)\setminus (C\cup L)$ is the total space of a punctured $\PP^1$-bundle over a punctured $\PP^1$ admitting a complete Riemannian metric of nonpositive sectional curvature. The second part of this theorem relies on a result of Robb \cite{robb1997branch} concerning fundamental groups of the complements to branching curves of generic complete intersections. 

In the negative direction, our obstruction for a group to be CAT(0) comes from the previous work of the third author \cite{zhu2023conditions}, generalizing methods of \cite{LIP21}.  We consider examples of quasi-projective varieties whose fundamental groups are extensions of free groups of finite rank. First, we exhibit a free-by-free group with infinite monodromy, which is not CAT(0), derived from the union of a 3-cuspidal quartic curve and a collection of lines (Theorem \ref{InfiniteMono}). This braid monodromy was first studied by Catanese--Wajnryb  \cite{catanese20043}, and in this case, we compute the monodromy explicitly. 

The second set of non-CAT(0) examples arise as monodromy groups associated with simple singularities of curves.  Such singularities are enumerated by root systems $R$ of type $A_n$ ($n\geq 1)$, $D_n$ ($n\geq 4$) or $E_n$ $(6\leq n\leq 8$). If $U_R$ is the corresponding semi-universal deformation space, then $\pi_1(U_R)$ is isomorphic to the Artin group $A_R$. The deformation space $U_R$ is the base space of a fiber bundle $V_R$ whose fibers are smooth affine curves. Consequently, $\pi_1(V_R)$ is an extension of $A_R$ by a free group of finite rank.

\begin{mainthm}\label{thm:mainB}
    Let $R$ be one of the root systems $E_6$, $E_7$, or $E_8$ and let $V_R$ be the total space of the affine family over the semi-universal deformation of the simple singularity of type $R.$ Then $\pi_1(V_R)$ is not CAT(0).
\end{mainthm}

A crucial ingredient here is a result of Wajnryb \cite{wajnryb1999artin}, which states that for these three root systems, the geometric monodromy is not injective. In contrast, Perron--Vannier \cite{PerronVannier} showed that when $R=A_n$ or $D_n$, the geometric monodromy is injective. At this time, the authors do not know whether $\pi_1(V_R)$ is CAT(0) in these cases.\\

\noindent\textbf{Acknowledgments:} The third author would like to thank Matthew Durham, Jose Ignacio Cogolludo-Agustin, and Chenxi Wu for useful discussions.

\section{Preliminaries on plane curve complements and CAT(0) groups}

In this section, we review some necessary material on computing fundamental groups of plane curve complements. We refer the reader to the Appendix for a more comprehensive overview of the topology of plane curves. We also recall basic definitions and properties of CAT(0) spaces and groups, introduce the LIP property, and describe the main obstruction we will use to show certain group extensions are not CAT(0).  

\subsection{The Zariski--van Kampen Theorem}\label{sec:ZVK}
\text{}\\
(See Appendices \ref{A1} and \ref{A2} and \cite{libgober2021complements} for additional details). The blow-up $\Bl_O\PP^2$ of $\mathbb P^2$ at a point $O$ admits a fibration over $\mathbb P^1$ with $\PP^1$-fibers. Given a plane curve $C$ and $O\not\in C$, we obtain a (singular) fibration from the complement of the plane curve $\Bl_O(\PP^2)\setminus C\to \PP^1$ with finitely many singular fibers, and whose generic fibers are complements $\PP^1\setminus \{x_1,...,x_d\}$ to a finite set of points of cardinality $d$, where $d=\deg(C)$. Let $L=\cup_{i=1}^nL_i$ be the union of lines corresponding to the singular fibers. After removing $L$, we get a (locally trivial) punctured sphere bundle over a punctured sphere: \begin{equation}\label{eqn:Blowup-Fibration}\PP^1\setminus \{x_1,...,x_d\}\to \Bl_O(\PP^2)\setminus (C\cup L)\rightarrow \PP^1\setminus \{y_1,...,y_n\},\end{equation} where $n$ is the number of singular fibers. 

If $n\geq 1$, the long exact sequence of homotopy groups for (\ref{eqn:Blowup-Fibration}) gives rise to a short exact sequence on $\pi_1$: 
\[1\to \pi_1(\PP^1\setminus \{x_1,...,x_d\})\to \pi_1(\Bl_O(\PP^2)\setminus (C\cup L))\to \pi_1(\PP^1\setminus \{y_1,...,y_n\})\to 1.\]
We have $\pi_1(\PP^1\setminus \{x_1,...,x_d\})\cong F_{d-1}$ and $\pi_1(\PP^1\setminus \{y_1,...,y_n\})\cong F_{n-1}$, where $F_k$ denotes the free group on $k$ generators. Substituting in the sequence above yields
\begin{equation}
    1\to F_{d-1}\to \pi_1(\Bl_O(\PP^2)\setminus (C\cup L))\to F_{n-1}\to 1
\end{equation}
This short exact sequence is actually split since $F_{n-1}$ is free. Thus the fundamental group of the complement of $C\cup L$ is a semi-direct product:
$$\pi_1(\Bl_O(\PP^2)\setminus (C\cup L))\cong F_{d-1} \rtimes_{\rho}  F_{n-1} ,$$ where $\rho$ is given by a lift of the induced $\pi_1$-representation $\rho\colon F_{n-1}\to \Out(F_{d-1})$.

If we rewrite $\pi_1(\PP^1\setminus \{x_1,\ldots,x_d\})$ as $\langle g_1,\ldots,g_d|\prod_{i=1}^d g_i=1\rangle$, and $\pi_1(\PP^1\setminus \{y_1,\ldots,y_n\})$ as $\langle t_1,..,t_n|\prod_{i=1}^n t_i=1\rangle$ then we get a presentation
$$\pi_1(\Bl_O(\PP^2)\setminus (C\cup L))\cong\left\langle g_i,t_j,~ 1\leq i\leq d,~1\leq j\leq n~\bigg \vert~\prod_{i=1}^d g_i=1,\prod_{j=1}^n t_j=1, g_i^{t_j}=\rho(t_j)(g_i)\right\rangle.$$
\begin{remark}
    Here, and throughout, for any group $G$ and elements $g,h\in G$, we will use the notation $h^g:=ghg^{-1}$; and for the inverse of $g$, we will use the notation $\overline g$.
\end{remark}

Filling in the singular fibers is equivalent to killing the generators $t_j$ on the level of $\pi_1$. On the other hand, since $O\notin C$ blowing up at $O$ does not affect the fundamental group of the complement, \emph{i.e.} $\pi_1(\Bl_O(\PP^2)\setminus C)\cong \pi_1(\PP^2\setminus C)$. It follows that $\pi_1(\PP^2\setminus C)$ has presentation: 
$$\pi_1(\PP^2\setminus C)\cong\left\langle g_i,~1\leq i\leq d~\bigg \vert~\prod_{i=1}^d g_i=1, g_i=\rho(t_j)(g_i),~1\leq j\leq n\right\rangle.$$
 Since $\Bl_O(\PP^2)\setminus (C\cup L)$ is the total space of a fiber bundle, we also have a monodromy homomorphism  $\phi:\pi_1(\PP^1\setminus \{y_1,...,y_n\})\to \Mod(\PP^1\setminus \{x_1,...,x_d\})$, where for any finite type surface $S$, $\Mod(S)$ denotes the \emph{mapping class group of $S$}. The representation $\rho$ defined above factors through $\phi$, thus restricting the possible free-by-free groups arising as $\pi_1(\Bl_O(\PP^2)\setminus (C\cup L)).$ In Section \ref{sec:cuspidal-quartic}, we will present an example of the computation of the monodromy homomorphism and fundamental group of the complement along the lines of the above description.

\subsection{CAT(0) groups and the LIP property}
\text{}\\
Let $X$ be a proper geodesic metric space. Given any three points $p,q,r$ in $X$, we can find a geodesic triangle $\Delta=\Delta(p,q,r)$ with $p,q,r$ as its vertices. By the triangle inequality, there exists a unique triangle $\Delta'=\Delta'(p',q',r')\subset \R^2$ with the same side lengths as $\Delta$, and a map $f\colon \Delta\rightarrow \Delta'$ mapping $p, q,r$ to $p',q',r'$, respectively, and restricting to an isometry on each edge. We call $\Delta'$ a \emph{comparison triangle} and $f$ a \emph{comparison map}.

\begin{definition} A proper, geodesic metric space $X$ is called CAT(0) if the following holds. Given a geodesic triangle $\Delta=\Delta(p,q,r)\subset X$ with  comparison triangle $\Delta'=\Delta'(p',q',r')\subset \R^2$ and comparison map $f\colon \Delta\rightarrow \Delta'$, for any $t_1\in [p,q]$ and $t_2\in [p,r]$ we have \[d_X(t_1,t_2)\leq d_{\R^2}(f(t_1),f(t_2)).\]

\end{definition}

The notion of a CAT(0) space generalizes the concept of nonpositive curvature to more general metric spaces.  Many results, such as the Cartan--Hadamard theorem, extend naturally to CAT(0) spaces.  In particular, CAT(0) spaces are geodesically convex and therefore contractible.  As in the case of curvature, being CAT(0) is a local notion and comes with an analogous local-to-global principle. Thus, if a space is simply connected and locally CAT(0), it is globally CAT(0). It is not difficult to show that CAT(0) spaces, the product of two CAT(0) spaces with the $\ell^2$-metric is again CAT(0) \cite{Bridson-Haefliger99}.

\begin{definition}
    A group $G$ is CAT(0) if it acts properly and cocompactly on a CAT(0) space by isometries.
\end{definition}

Examples of CAT(0) spaces include simplicial trees and nonpositively curved Riemannian manifolds. Accordingly, examples of CAT(0) groups include free groups, free abelian groups, and hyperbolic surface groups. Conjecturally, all Artin groups are CAT(0) as well, but this has only been verified in specific cases and is open even for finite type Artin groups such as braid groups $B_n$, $n> 7.$ 


Llosa-Isenrich--Py \cite{LIP21} studied when a surface-by-surface group is CAT(0) if both base and fiber genus are at least 2. They derived a necessary condition on the algebraic monodromy: it is either injective or has finite image. In \cite{zhu2023conditions}, the third author extended their results by introducing the following definition and general obstruction for an extension to be CAT(0).

\begin{definition}[Property LIP]\label{LIP}
	A group $Q$ has \emph{Property LIP} if, for every infinite normal subgroup $N \trianglelefteq Q$, there is an infinite finitely generated subgroup $N_0<N$ so that the centralizer $C_Q(N_0)$ is finite.
\end{definition}


\begin{theorem}[Theorem 1.5, \cite{zhu2023conditions}]\label{thm:LIP-Obstruction}
    Suppose $G$ fits into a short exact sequence $$1\to R \to G \to\Gamma\to 1$$ where $R$ is finitely generated and has trivial center, and $Q$ has Property LIP. Then if $G$ is $\mathrm{CAT}(0)$, the algebraic monodromy  $\rho\colon Q\to\mathrm{Out}(R)$ either has finite image or finite kernel.
\end{theorem}

\begin{remark}
    The initial version of this theorem assumed that $R$ is a surface group, but later, the author realized that the proof in \cite{zhu2023conditions} only requires $R$ to be finitely generated (see \cite{zhu2023conditions}-p.4649, the eighth line from the bottom), have a trivial center (see \cite[Remark 3.2]{zhu2023conditions} and the proof of \cite[Lemma 30]{LIP21} to make the theorem work.
\end{remark}

Although the definition of Property LIP is easy to state, a simple criterion which guarantees it holds is provided by the following.

\begin{proposition}[Theorem 1.8, \cite{zhu2023conditions}]\label{prop:AC-is-LIP}
    If $Q$ is acylindrically hyperbolic, then $Q$ has Property LIP.
\end{proposition}

We recall that a group is called acylindrically hyperbolic if it admits a nonelementary acylindrical action on a hyperbolic geodesic metric space. Since the precise definition of acylindrical action will not figure prominently in the sequel, we will not review it here, and instead refer the reader to \cite{Osin16}.

\section{Positive results}\label{finitemonodromy}
In this section we explore several infinite families of curves $C$ for which $\pi_1(\PP^2\setminus (C\cup L))$ or $\pi_1(\PP^2\setminus C)$ is CAT(0). The first class consists of those with finite monodromy. For these, we show that $\pi_1(\PP^2\setminus (C\cup L))$ is CAT(0) and moreover that the total space $\Bl_O(\PP^2\setminus (C\cup L)$ in fact admits a complete, finite-volume nonpositively curved Riemannian metric. The second class of examples is derived from Robb's computation of the fundamental group of a branching curve for a generic projection of a complete intersection surface. We will deduce Theorem \ref{thm:mainA} as a combination of Proposition \ref{prop:Finite-Monodromy} and Theorem \ref{thm:Complete-Intersection-Abelian}.

\subsection{Finite monodromy}
\text{}\\
As a consequence of Theorem \ref{thm:LIP-Obstruction}, the existence of a CAT(0) extension group results in a dichotomy for the monodromy, namely that it either be finite or have finite kernel.  In this section, we investigate the fundamental groups of plane curve complements with finite monodromy. 

\begin{example}
For $n\geq 2$, consider the family of plane curves given by the equation  \[x^{n-1}z-y^n=0.\] This curve has a single multiple cusp of order $n$ at $[0:0:1]$. The projection onto $x$ has one singular fiber at $x=0$, hence the base of fibration is a twice-punctured sphere 
while the fiber is an $n$-punctured sphere. This yields the extension \[1 \rightarrow F_{n-1} \rightarrow G \rightarrow \mathbb Z \rightarrow 1\]
with monodromy isomorphic $\Z/n\Z$, and acts by cyclically permuting the punctures of the fiber.
\end{example}

For plane curves with finite monodromy we prove:
\begin{proposition}\label{prop:Finite-Monodromy}
    Let $C\subseteq \PP^2$ be a plane curve and suppose $\pi_1(\Bl_O(\PP^2)\setminus C)$ fits into a short exact sequence \begin{equation*}
        1\to F_n\to\pi_1(\Bl_O(\PP^2)\setminus C)\to F_m\to 1
    \end{equation*}
    with finite monodromy. Then $\pi_1(\Bl_O(\PP^2)\setminus C)$ is CAT(0). Moreover, $\Bl_O(\PP^2\setminus C)$ admits a finite-volume Riemannian metric of nonpositive curvature.
\end{proposition}

This will be an easy consequence of the next two lemmas.
\begin{lemma}\label{lem:Finite-Free-By-Free}Suppose $G$ fits into a short exact sequence\[1\to F_n\to G\to F_m\to 1\]with finite monodromy $H\leq \Out(F_n)$. Then $G$ is the fundamental group of a compact, NPC 2-complex $X$ whose universal cover is a product of trees.
\end{lemma}
\begin{proof}By the (Nielsen) realization theorem for $\text{Out}(F_n)$ due independently to Culler \cite{culler1984finite}, Khramtsov \cite{khramtsov1985finite} and Zimmermann \cite{zimmermann1981homoomorphismen}, there exists a rank $n$ graph $\Gamma$ realizing $H$ by isometries (here the metric assigns a fixed positive length $\lambda$ to each edge of $\Gamma$). Let $R_m$ be a wedge of $m$ circles, and let $T_m$ be the universal cover of $R_m.$ Consider the orthogonal product $X'=\Gamma\times T_m$, which is locally  CAT(0) since each of the factors are. Then $F_m$ acts on $X'$ by isometries via $H$ on the $\Gamma$-factor and by deck transformations on $T_m$.  The action is free, proper and cocompact, since it is on the second factor and $\Gamma$ is compact. Thus, the quotient $X$ is locally CAT(0) and has a fundamental group of $G$ by construction. The universal cover of $X$ is $\widetilde{\Gamma}\times T_m$, which is a product of trees. 
\end{proof}

\begin{lemma}\label{lem:Classifying-Sphere-Bundle}
    Let $X$ be the total space of locally trivial $S_{0,k}$-bundle over $S_{0,l}$ with $l\geq 1$.  Then $X$ is determined up to equivalence by the monodromy homomorphism $\rho\colon \pi_1(S_{0,l})\rightarrow \Mod(S_{0,k})$.
\end{lemma}
\begin{proof}
Since $S_{0,l}$ is homotopy equivalent to a wedge of $l-1$ circles, any such bundle is determined up to equivalence, via the clutching construction, by a homomorphism $\pi_1(S_{0,l})\rightarrow \pi_0\Diff(S^2,\{p_1,\ldots p_k\}):=\Mod(S_{0,k})$.
\end{proof}

We are now ready to prove Proposition \ref{prop:Finite-Monodromy}.

\begin{proof}[Proof of Proposition \ref{prop:Finite-Monodromy}]
That $\pi_1(\Bl_O(\PP^2)\setminus C)$ is CAT(0) follows directly from Lemma \ref{lem:Finite-Free-By-Free}. For the second part, we know that $\Bl_O(\PP^2\setminus C)$ is the total space of a locally trivial $S_{0,n+1}$-bundle over $S_{0,m+1}$, where $S_{0,k}$ denotes the $k$-punctured 2-sphere. The monodromy group $H$ is a finite subgroup of the mapping class group of $S_{0,n+1}$.  By the solution to the Nielsen realization problem for punctured surfaces \cite[Theorem 11.14]{Zieschang}, we can find a finite-volume hyperbolic surface $V$ homeomorphic to $S_{0,n+1}$ which realizes the action of $H$ by isometries. 

Now let $W$ be any finite-volume hyperbolic surface homeomorphic to $S_{0,m+1}$. The universal cover of both $S_{0,n+1}$ and $S_{0,m+1}$ can be identified with the hyperbolic plane $\Hy^2$.  Mimicking the construction in the proof of Lemma \ref{lem:Finite-Free-By-Free}, we consider the product $V\times \Hy^2$, and let $\pi_1(W)\cong F_m$ act by isometries on $V$ and by deck transformations in the second component.  Again the action is free and proper. The quotient $X$ is therefore a Riemannian manifold whose universal cover can be identified with $\Hy^2\times \Hy^2$, hence is nonpositively curved. 

By construction, $X$ is the total space of an isotrivial $S_{0,n+1}$-bundle over $S_{0,m+1}$. In particular, $X$ is finitely covered by a product $V\times W'$ where $W'$ is the finite cover of $W$ corresponding to the kernel of the monodromy homomorphism. Since $V$ and $W$ have a finite volume, this implies $V\times W'$ has finite volume, and thus that $X$ does as well. The fact that $X$ is homeomorphic to $\Bl_O(\PP^2\setminus C)$ follows from Lemma \ref{lem:Classifying-Sphere-Bundle}.
\end{proof}

\subsection{Branch loci of complete intersections}
\text{}\\
For $n\geq 3$, let $X\subset \PP^n$ be a smooth surface and consider a projection $\pi\colon X\rightarrow \PP^2$. This will exhibit $X$ as a branched cover over $\PP^2$, branched along a singular curve $C\subset \PP^2$.  For a generic projection, $C$ will have only cusps and nodes. 

The branching curve of a generic projection of a smooth cubic surface in $\PP^3$ is a sextic with six cusps on a conic. Zariski \cite{zariski1929problem} showed that in this case the fundamental group of the complement  is $\pi_1(\PP^2\setminus C)\cong \Z/2*\Z/3\cong \PSL_2(\Z)$.  Zariski famously used this calculation to show that the moduli space of irreducible sextics with six cusps in $\PP^2$ is disconnected (for sextic curves with six cusps not on a conic, 
the fundamental group is $\mathbb Z/6$).

Moishezon \cite{Moishezon81} generalized Zariski's result by calculating the analogous fundamental group for generic projections 
of a smooth hypersurface in $\PP^3$ of arbitrary degree $d$. For these, the complement of the branch curve $C$ has fundamental group $\pi_1(\PP^2\setminus C)\cong B_d/Z(B_d)$, where $B_d$ is the braid group of the disk on $d$ strands, and $Z(B_d)$ is its center. This quotient can be identified with a finite index subgroup of the mapping class group of $S_{0,d+1}$. 

Recall that the braid group $B_d$ maps onto the symmetric group $S_d$ with kernel the pure braid group $P_d$. The abelianization of $B_d$ is $\Z$, and is obtained by sending each of the standard positive Artin generators to $1\in \Z$.  The center $Z(B_d)\cong \Z$ is generated by the square of the ``long braid," and since it is a positive braid, it maps to $d(d-1)\in \Z$. Moreover, it is itself a pure braid, and after passing to $P_d$ we have a splitting $P_d\cong P_{0,d}\times Z(B_d)$ where $P_{0,d}$ is the kernel of the restriction of the abelianization map of $B_d$ to $P_d$. Thus the image of $P_d$ in the quotient $B_d/Z(B_d)$ is simply $P_{0,d}$.

Robb extended Moishezon's result to the case where $X\subset \PP^n$ is a smooth, complete intersection of codimension greater than 1 \cite{robb1997branch}. Let $x_1,\ldots, x_{d-1}$ be the standard generators of $B_d$.  For $d\geq 4$, define \begin{equation}
    \tB_d:=B_d/\left\llangle [x_2,(x_3x_1)^{-1}x_2(x_3x_1)]\right\rrangle
\end{equation}
where $\llangle\cdot\rrangle$ indicates normal closure. Robb shows that the center of $\tB_d$ is isomorphic to $\Z\oplus \Z/2$. Let $\eta$ and $\mu$ be the generators of $\Z$ and $\Z/2$, respectively. By the adjunction formula and Riemann-Hurwitz, the degree $\deg(S)$ will be even and a multiple of $d=\deg(X)$. Define an integer \[m=\left\{\begin{array}{cl}\frac{\deg(S)}{d}&\text{$d$ even}\\\frac{\deg(S)}{2d}& \text{$d$ odd}\end{array}\right.\]
Robb's extension of Zariski and Moishezon's result is the following.
\begin{theorem}[\cite{robb1997branch}]\label{thm:Robb}
    Let $X\subset \PP^n$ be a smooth, non-degenerate complete intersection of codimension greater than 1. Suppose $C\subset \PP^2$ is the branch locus of a generic projection $X\rightarrow \PP^2$, and let $C_a$ be a generic affine portion of $C$. Then for some $e\in \{0,1\}$,\[\pi_1(\C^2\setminus C_a)\cong \tB_d \textrm{ and } \pi_1(\PP^2\setminus C)\cong \tB_d/\langle \eta^m\mu^e\rangle .\]
\end{theorem}

One knows that $\PSL_2(\Z)$ is CAT(0); indeed it is virtually free. Conjecturally, the braid group $B_d$ for all $d\geq 1$ is CAT(0), and this has been verified for $d\leq 7$ \cite{BradyMcCammond,Haettel-Kielak-Schwer,jeong2023seven}. A result of Bowers--Ruane \cite{BowersRuane} implies that if $G\times \mathbb Z^n$ is CAT(0), then $G$ is CAT(0). Hence, by the remarks above concerning $P_d$ and $Z(B_d)$, we see that if $B_d$ is CAT(0) then $B_d/Z(B_d)$ is as well. We now show that both of the groups $\widetilde{B}_d$ and $\widetilde{B}_d/\langle \eta^m\mu^e\rangle$ appearing in Theorem \ref{thm:Robb} is CAT(0).  In fact, we will prove:

\begin{theorem}\label{thm:Complete-Intersection-Abelian}
    For any $d\geq 4$, $\tB_d$ is virtually abelian.  In particular, both groups $\pi_1(\C^2\setminus C_a)$ and $\pi_1(\PP^2\setminus C)$ appearing in Theorem \ref{thm:Robb} are CAT(0).
\end{theorem}

The second statement will be implied by the first, as a consequence of the following easy lemma. 

\begin{lemma}\label{lem:VAbelian-Is-CAT(0)} Suppose $G$ is finitely generated and virtually abelian. Then $G$ acts properly and cocompactly on a CAT(0) space.
\end{lemma}
\begin{proof} Let $H\unlhd G$ be a finite index normal subgroup isomorphic to $\Z^n$. Then $G$ fits into a short exact sequence \begin{equation}\label{eqn:VAbelian-Extension}
    1\to H\cong\Z^n\to G\xrightarrow[]{\varphi} F\to 1,
\end{equation}
where $F$ is a finite group. From the extension in (\ref{eqn:VAbelian-Extension}), we obtain an algebraic monodromy homomorphism $\rho\colon F\rightarrow \GL_n(\Z)$, making $\Z^n$ into an $\Z[F]$-module.  Let $F'$ be the image of $\rho$. The space of marked, flat $n$-dimensional tori is the symmetric space $X_n:=\SL_n(\R)/\SO_n(\R)$. Since $F'$ is a finite subgroup of $\GL_n(\Z)$ and $X_n$ is a CAT(0) space, we can find a marked, flat $n$-torus $T$ on which $F'$ acts by isometries. Thus we obtain an explicit lattice $\Lambda\leq \R^n$ and a homomorphism $\hat{\rho}\colon G\rightarrow N(\Lambda)$, where the latter is the normalizer of $\Lambda$ in $\Isom(\R^n)$. The image $G'$ of $\hat{\rho}$ acts effectively on $\R^n$ by isometries.  

On the other hand, the kernel of $\hat{\rho}$ is a finite normal subgroup $K\unlhd G$, hence $K\cap H=\{1\}$, and $K$ maps injectively to $F$.  Choose a sufficiently high-dimensional closed 
Euclidean disk $D^k$  on which $F$ acts effectively by isometries and consider the product $X=\R^n\times D^k$. Then $G$ acts via $\hat{\rho}$ on $\R^n$ and by factoring through $\varphi$ on $D^k$. This action is effective by construction, and proper since each point stabilizer is conjugate to subgroup of $F$.
\end{proof}

\begin{proof}[Proof of Theorem \ref{thm:Complete-Intersection-Abelian}]
Since $\tB_d$ is the quotient of $B_d$ by a commutator, the abelianization of $B_d$ factors through $\tB_d$.  Hence, their abelianizations are equal and isomorphic to $\Z$. Let $\tP_d\leq \tB_d$ be the finite-index subgroup which is the image of $P_d$ under the quotient. We know that $P_d$ splits as a product $P_{0,d}\times Z(B_d)\cong P_{0,d}\times \Z$, where $P_{0,d}$ lies in the kernel of the abelianization map. It follows that $\tP_{d}$ also splits as a product $\tP_{0,d}\times \Z$ where $\Z$ lies in the center $Z(\tP_d)$. Since $\tP_d$ has finite index in $\tB_d$, it is enough to show that $\tP_{0,d}$ is virtually abelian by Lemma \ref{lem:VAbelian-Is-CAT(0)}.

Robb \cite[Proposition 2.2]{robb1997branch} proves that $\tP_{0,d}$ is generated by elements $v_1,\ldots, v_{d-1},\mu$ subject to relations 
\begin{itemize}
\item $\mu^2=1$.
\item For all $i\neq j$, $[v_i,v_j]=\left\{\begin{array}{cc}
   1  & |i-j|\geq 2 \\
\mu& |i-j|=1
\end{array}\right.$
\item $[\mu,v_i]=1$ for all $i$.
\end{itemize}
This implies $\tP_{0,d}$ is a central extension
\[1\to \Z/2\to \tP_{0,d}\to \Z^{d-1}\to 1\]
Since $\mu$ has order 2, the subgroup generated by the squares of the $v_i$ and $\mu$ has finite index and is isomorphic to $\Z^{d-1}\oplus \Z/2$.  Thus $\tP_{0,d}$, and hence $\tB_d$, is virtually free abelian.
\end{proof}

\section{Negative Results}
We now turn to examples of quasi-projective groups that are not CAT(0). First, we present a detailed example showing that the complement of the 3-cuspidal quartic together with  the singular fibers under a (non-generic) projection does not have CAT(0) fundamental group. In general, complements of \emph{reducible} plane curves will not be CAT(0), and this is an example of this phenomenon. Second, we consider the family of curves arising from the universal deformation of an isolated plane curve singularity of ADE type. In the case of $E_6$, $E_7$, and $E_8$, we use a result of Wajnryb to show that the fundamental group of the universal family is not CAT(0), proving Theorem \ref{thm:mainB}.

\subsection{The real 3-cuspidal quartic}\label{sec:cuspidal-quartic}
\text{}\\
If $d\le 3$,  $\Mod(\PP^1\setminus \{x_1,...,x_d\})$ is finite, hence the image of the monodromy homomorphism is necessarily finite. In this section, we will closely examine an explicit example of a 3-cuspidal quartic (\emph{i.e.} $d=4$) curve $C$ whose braid monodromy is infinite. Although $\pi_1(\Bl_O(\PP^2)\setminus C)=\pi_1(\PP^2\setminus C)$ will be finite and thus CAT(0), we will nevertheless show that $\pi_1(\Bl_O(\PP^2)\setminus (C\cup L))$ is not CAT(0), where $L $ is the union of singular lines arising from a particular non-generic projection.  We closely follow the exposition of the real 3-cuspidal quartic curve from Catanese and Wajnryb, \cite[Section 4]{catanese20043}.

A \emph{3-cuspidal quartic} is an irreducible plane curve of degree $4$ having 3 ordinary cusps as singularities (see Example \ref{ex:3cuspidalquartic} in the Appendix for a more complete description of its construction). Any two such curves are equivalent up to equisingular deformation, and Zariski showed that the fundamental group of the complement of the 3-cuspidal quartic is $B_3(S^2)$, the braid group of the sphere on 3 strands. If $C\subset \mathbb \PP^2$ is an irreducible quartic curve that is not a 3-cuspidal quartic, then the fundamental group $\pi_1(\mathbb \PP^2\setminus C)$ is abelian (cf. \cite[(4.3)]{dimca2012singularities}).

\begin{example}[\cite{catanese20043}, Section 4]\label{exampleCata}
The following is a real quartic with three cusps: \[C= \{F([x:y:z]=(x^2 + y^2)^2 + x^3z + 9xy^2z +\frac{27}{4}y^2z^2 = 0\}.\]

Since $O:=[0:1:0]\not\in C$,  we can define a projection $\phi_C\colon C\rightarrow \PP^1$ after blowing up at $O$, which projects $[x:y:z]\mapsto [x:z]$ away from $O$. For any $t\in \PP^1$, $\phi_C$ is ramified at $t$ if and only if there exists $p\in C_t$ such that $\frac{\partial F}{\partial y}(p)=0$.
 There are two such ramification points: a double tangency at $y_4=[1:0]$ and a single tangency $y_2=[-1:1]$. There are $2$ more singular fibers from $y_1=[\frac{-9}{8}:1]$ (double cusp), $y_4=[0:1]$ (cusp). 
 
 Let $L=\cup_{i=1}^4L_i$ be the union of the four singular fibers.  After deleting $L$, let  the generators of $\pi_1(\PP^1\setminus \{y_1,\cdots,y_4\})$ corresponding to $y_1$, $y_2$ and $y_3$ be $\tau_1$, $\tau_2$ and $\tau_3$, respectively. In the chart $z\neq 0$, we choose a basepoint on the positive real axis and orient these generators so that they follow a path in the upper half plane, go counterclockwise around exactly one of the punctures (all of which are real), and then go back to the basepoint in the upper half-plane. Since $C$ has degree 4, we have a fiber bundle $\PP^1\setminus \{x_1,...,x_4\}\to \Bl_O(\PP^2)\setminus (C\cup L))\xrightarrow{\phi} \PP^1\setminus \{y_1,...,y_4\}$ as in the setting described above.

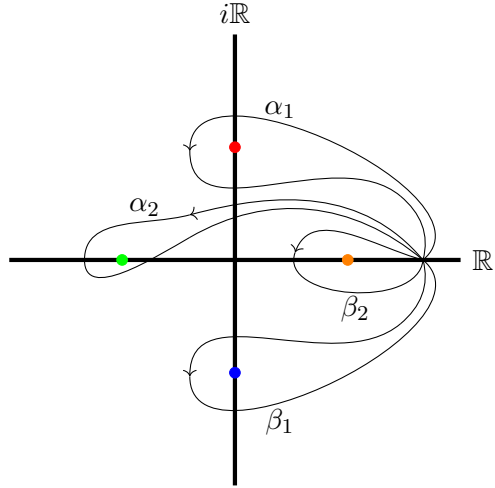
\begin{figure}[h]
    \centering
\begin{tikzpicture}

\draw[ultra thick] (-3,0)--(3,0) node[right]{$\mathbb R$};
\draw[ultra thick] (0,-3)--(0,3) node[above]{$i\mathbb R$};
\filldraw [orange] (1.5,0) circle (2pt);
\filldraw [green] (-1.5,0) circle (2pt);
\filldraw [red] (0,1.5) circle (2pt);
\filldraw [blue] (0,-1.5) circle (2pt);
\draw [->](2.5,0) to[out=40,in=90] (-0.6,1.45); 
\draw(-0.6,1.45) to[out=-90,in=140] (2.2,0.85); 
\draw(2.2,0.85) to[out=-40,in=75] (2.5,0); 
\draw (0.6,2) node  {$\alpha_1$};

\draw(2.5,0) to[out=-40,in=-90] (-0.6,-1.55); 
\draw[<-](-0.6,-1.55) to[out=90,in=-140] (2.2,-0.85); 
\draw(2.2,-0.85) to[out=40,in=-75] (2.5,0); 
\draw (0.6,-2.15) node  {$\beta_1$};

\draw[->] (2.5,0) to[out=165,in=70] (0.8,0.1);
\draw(0.8,.1) to[out=-110,in=-95] (2.5,0) ; 
 \draw (1.6,-0.65) node  {$\beta_2$};

\draw (2.5,0) to[out=140,in=30] (-0.6,0.3); 
\draw [->](2.5,0) to[out=130,in=15] (-0.6,0.6); 
\draw (-0.6,0.3) to[out=-150,in=-90] (-2,0); 
\draw (-0.6,0.6) to[out=-165,in=90] (-2,0); 
\draw (-1.2,0.7) node  {$\alpha_2$};

\end{tikzpicture}
\caption{A choice of generating set for the fundamental group of the fiber at a point along the positive real axis in the base. The punctures $x_1$, $x_2$, $x_3$ and $x_4$ correspond to $\alpha_1$, $\alpha_2$, $\beta_1$ and $\beta_2$, respectively.}
\label{fig:4loops}

\end{figure}

Along the positive real axis, the fiber contains two real punctures and two purely imaginary ones.  We order the punctures and choose representatives for simple loops around the pure imaginary punctures $\alpha_1$ for $x_1$, $\alpha_2$ for $x_2$, and the real punctures $\beta_1$ for  $x_3$, $\beta_2$ for $x_4$ as in Figure \ref{fig:4loops}. The $\tau_i$  act as the following braids in terms of the standard Artin generators $\{\sigma_i\}_{i=1}^3$:
\begin{align*}\tau_1&=\sigma_3\sigma_2^3\sigma_3^{-1}\\\tau_2&=\sigma_2\sigma_1\sigma_2^{-1}\\ \tau_3&=\sigma_1^3\sigma_3^3.\end{align*}
These braids are pictured in Figure \ref{fig:braids}. In terms of the chosen basis $\{\alpha_1,\alpha_2,\beta_1,\beta_2\}$ for $\pi_1(\PP^1\setminus \{x_1,...,x_4\})$ the algebraic monodromy is given by:

\noindent\begin{minipage}{.32\linewidth}
\begin{align*}
\tau_1\colon \alpha_1&\longmapsto \alpha_1 \\
\alpha_2&\longmapsto\alpha_2\beta_2\alpha_2\beta_2\overline{\alpha_2}\overline{\beta_2}\overline{\alpha_2} \\
\beta_1&\longmapsto \beta_1\\
\beta_2&\longmapsto \alpha_2\beta_2\alpha_2\overline{\beta_2}\overline{\alpha_2}\\
\end{align*}
\end{minipage}
\noindent\begin{minipage}{.32\linewidth}
\begin{align*}
\tau_2\colon \alpha_1&\longmapsto \alpha_1\beta_2\beta_1\overline{\beta_2}\overline{\alpha_1} \\
\alpha_2&\longmapsto\alpha_1\beta_2\overline{\beta_1}\overline{\beta_2}\alpha_2\beta_2\beta_1\overline{\beta_2}\overline{\alpha_1} \\
\beta_1&\longmapsto \overline{\beta_2}\alpha_1\beta_2\\
\beta_2&\longmapsto \beta_2\\
\end{align*}
\end{minipage}
\noindent\begin{minipage}{.32\linewidth}
\begin{align*}
\tau_3\colon \alpha_1&\longmapsto \alpha_1\alpha_2\alpha_1\alpha_2\overline{\alpha_2}\overline{\alpha_1} \\
\alpha_2&\longmapsto\alpha_1\alpha_2\alpha_1\overline{\alpha_1}\overline{\alpha_2}\overline{\alpha_1}\\
\beta_1&\longmapsto \beta_2\beta_1\beta_2\overline{\beta_1}\overline{\beta_2}\\
\beta_2&\longmapsto \beta_2\beta_1\beta_2\beta_1\overline{\beta_2}\overline{\beta_1}\overline{\beta_2}\\
\end{align*}
\end{minipage}
where, for ease of notation, we have used $\overline{\alpha}$ to denote $\alpha^{-1}$.
After an easy simplification, the fundamental group of the complement of the 3-cuspidal quartic curve, $\pi_1(\PP^2\setminus C)$, is 
$$\langle \alpha_1,\alpha_2|\alpha_2^{\alpha_1}=\alpha_1^{\alpha_2},\alpha_2\alpha_1^2\alpha_2=1\rangle\cong B_3(S^2),$$
which is finite, and, in particular, CAT(0). 
\end{example}

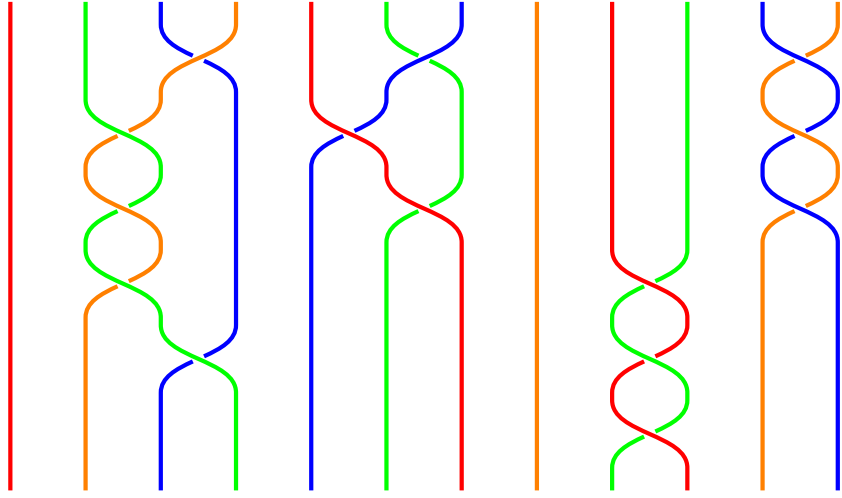
\begin{figure}

\begin{tikzpicture}[
braid/.cd,
number of strands=4,
every strand/.style={ultra thick},
strand 1/.style={red},
strand 2/.style={green},
strand 3/.style={blue},
strand 4/.style={orange},
]
\pic at (0,0) {braid={s_3^{-1} s_2 s_2 s_2 s_3 1 }};
 \pic at (4,0) {braid={s_2^{-1}s_1s_2 1 1 1}};
\pic at (8,0) {braid={s_3 s_3 s_3 s_1 s_1 s_1}};
\end{tikzpicture}

\caption{From left to right, the three 4-stranded braids correspond to $\tau_1$, $\tau_2$, and $\tau_3$. The strands of each braid correspond to $x_1$, $x_2$, $x_3$, $x_4$ in left to right order, and match the colors in Figure \ref{fig:4loops}.}
\label{fig:braids}
\end{figure}

\label{quarticcurve}
Combined with the results of Section \ref{sec:ZVK}, we obtain a presentation $\pi_1(\Bl_O(\PP^2)\setminus (C\cup L))=\langle S|R\rangle,$ where
\begin{align*}S&=\{\alpha_1,\alpha_2,\beta_1,\beta_2,\tau_1,\tau_2,\tau_3\},\\

  R &= \left\{ \begin{array}{l}
    \alpha_1^{\tau_1}=\alpha_1, \alpha_1^{\tau_2}=\beta_1^{\alpha_1\beta_2},\alpha_1^{\tau_3}=\alpha_2^{\alpha_1\alpha_2\alpha_1}; \\
    
    \alpha_2^{\tau_1}=\beta_2^{\alpha_2\beta_2\alpha_2}, \alpha_2^{\tau_2}=\alpha_2^{\alpha_1\beta_2\beta_1\overline{\beta_2}},\alpha_2^{\tau_3}=\alpha_1^{\alpha_1\alpha_2};\\
    
\beta_1^{\tau_1}=\beta_1,\beta_1^{\tau_2}=\alpha_1^{\overline{\beta_2}},\beta_1^{\tau_3}=\beta_2^{\beta_1\beta_2};\\

\beta_2^{\tau_1}=\alpha_2^{\alpha_2\beta_2},\beta_2^{\tau_2}=\beta_2,\beta_2^{\tau_3}=\beta_1^{\beta_2\beta_1\beta_2};\\
\alpha_1\alpha_2\beta_1\beta_2=1.
  \end{array}\right\}
  \end{align*}

Solving for  $\beta_1=\overline{\beta_2}\overline{\alpha_2} \overline{\alpha_1}$, we can simplify the above presentation to the following: 
\begin{align*}S&=\{\alpha_1,\alpha_2,\beta_2,\tau_1,\tau_2,\tau_3\},\\

  R &= \left\{ \begin{array}{l}
   \alpha_1^{\tau_1}=\alpha_1, \alpha_1^{\tau_2}=(\overline\beta_2\overline\alpha_2\overline\alpha_1)^{\alpha_1\beta_2},\alpha_1^{\tau_3}=\alpha_2^{\alpha_1\alpha_2\alpha_1}; \\
   
   \alpha_2^{\tau_1}=\beta_2^{\alpha_2\beta_2\alpha_2}, \alpha_2^{\tau_2}=\alpha_2^{\alpha_1\overline\alpha_2\overline\alpha_1\overline{\beta_2}},\alpha_2^{\tau_3}=\alpha_1^{\alpha_1\alpha_2};\\
   
\beta_2^{\tau_1}=\alpha_2^{\alpha_2\beta_2},\beta_2^{\tau_2}=\beta_2,\beta_2^{\tau_3}=(\overline\beta_2\overline\alpha_2\overline\alpha_1)^{\overline\alpha_2\overline\alpha_1\beta_2}
  \end{array}\right\}
\end{align*}

Now we consider the geometric monodromy of the associated punctured sphere bundle $$\PP^1\setminus \{x_1,...,x_4\}\to \Bl_O(\PP^2)\setminus (C\cup L)\xrightarrow{\phi} \PP^1\setminus \{y_1,...,y_4\}.$$


By \cite[Proposition 2.7]{farb2011primer}, there is an isomorphism \[\Mod(S_{0,4})= (\Z/2\Z\times \Z/2\Z)\rtimes \PSL_2(\Z),\] while \[B_4(S^2)\cong(\Z/2\Z\times \Z/2\Z)\rtimes \SL_2(\Z)\]

Let $q\colon T^2\to T^2/\langle\iota\rangle=S_{0,4}$ be the quotient map, where $\iota$ is the hyperelliptic involution. Identify the torus $T^2$ with a quotient of the unit square (rotated by $\pi/4$) as in Figure \ref{fig:involution}. With respect to the usual group structure on $T^2$, the involution $\iota$ can be identified with $-I$ and the ramification locus on $T^2$ consists of the four points $\tilde A_1=(1/2,1/2)^T,\tilde A_2=(0,1/2),\tilde B_1=(0,0)^T,\tilde B_2=(1/2,0)^T$, which cover the branch points $x_1,x_2,x_3,x_4$ as in Figure \ref{fig:4loops}. 

These points generate the 2-torsion subgroup $\Z/2\Z\times \Z/2\Z\leq T^2$, and are preserved by the action of $\SL_2(\Z)$. Translations of the torus by these four elements comprise the normal Klein four subgroup in the semi-direct product decompositions of $\Mod(S_{0,4})$ and $B_4(S^2)$ above.  
Thus, the product of two elements $(v_1,M_1), (v_2,M_2)\in B_4(S^2)\cong(\mathbb Z/2\mathbb Z\times \mathbb Z/2\mathbb Z)\rtimes \SL(2,\mathbb Z)$ is given by 
$$(v_1,M_1)\cdot (v_2,M_2)=(v_1+M_1\cdot v_2,M_1 M_2)$$ where $A\in \SL_2(\Z)$ acts on $v\in \Z/2\Z\times\Z/2\Z$ via reduction mod 2.

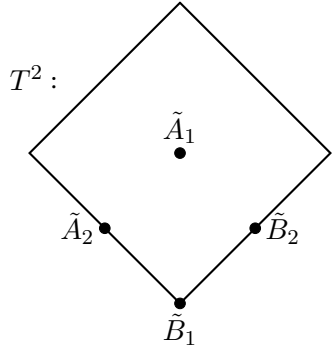
\begin{figure}
    \centering

    \begin{tikzpicture}[line cap=round,line join=round]
        \draw[thick] (-2,0) -- (0,2);
	\draw[thick] (0,2) -- (2,0);
        \draw[thick] (-2,0) -- (0,-2);
        \draw[thick] (0,-2) -- (2,0);
        \filldraw [] (1,-1) circle (2pt);
	\filldraw [] (0,0) circle (2pt);
	\filldraw [] (-1,-1) circle (2pt);
	\filldraw [] (0,-2) circle (2pt);
   \draw[] (0,-2) node [below] {$\tilde B_1$};
   \draw[] (0,0) node [above] {$\tilde A_1$};
   \draw[] (-1,-1) node [left] {$\tilde A_2$};
   \draw[] (1,-1) node [right] {$\tilde B_2$};
   \draw[] (-1.5,1) node [left] {$T^2:$};
\end{tikzpicture}
    
    \caption{Representation of $T^2$ with the 4 ramification points of $\iota$ at $\tilde{A}_1$, $\tilde{A}_2$, $\tilde{B}_1$, and $\tilde{B}_2$. }
    \label{fig:involution}
\end{figure}

The generators of $\SL_2(\Z)$ are the two Dehn twists\[
S=\begin{pmatrix}
1 & 1 \\
0 & 1 
\end{pmatrix}, T=\begin{pmatrix}
1 & 0 \\
1 & 1 
\end{pmatrix},\] both of which fix $\tilde B_1$. We can relate these to the standard Artin generators of $B_4(S^2)$ as follows. Each positive half-twist generator in $S_{0,4}$ lifts to a map of $T^2$  which is isotopic one of these two Dehn twists, but fixing different pairs of the ramification points. To fix notation, let \[
x:=\begin{pmatrix}
\frac{1}{2}  \\
0 
\end{pmatrix},
y:=\begin{pmatrix}
0  \\
\frac{1}{2}
\end{pmatrix}\]
be the two generators of $\Z/2\Z\times\Z/2\Z$, thought of translations on $T^2$. The three generators $\sigma_1,\sigma_2$, $\sigma_3$ then correspond to the following pairs group elements 
\begin{align*}
	\sigma_1&=(0,S) \\
	\sigma_2&=(x,I)\cdot(0,T)\cdot(x,I)=(x,T)\cdot(x,I) =( x+y+x,T)=( y,T)\\

\sigma_3&=(y,I)\cdot( 0,S)\cdot(y,I)=( y,S)\cdot(y,I)=( y+x+y,S)=( x,S)
\end{align*}

Now we can compute the image of the three generators of the fundamental group of the base space, $\tau_1,\tau_2,\tau_3$, in the braid group $B_4(S^2)$, in terms of matrices:

\begin{align*}
\tau_1=\sigma_3\sigma_2^3\sigma_3^{-1}&=(x,S)\cdot
(y,T^3)\cdot
(x,S^{-1})\\

&=(x+x+y,ST^3)\cdot(x,S^{-1})\\

&=(y+y,ST^3S^{-1})
=(0,ST^3S^{-1})\\

\tau_2=\sigma_2\sigma_1\sigma_2^{-1}&=(y,T)\cdot(0,S)\cdot
(y,T^{-1})\\

&=(y,TS)\cdot(y,T^{-1})=(x+y,TST^{-1})\\

\tau_3=\sigma_1^3\sigma_3^3&=(0,S^3)\cdot(x,S^3)
=(x,S^6)
\end{align*}

Recall that $\SL_2(\Z)\cong \Z/4\Z\ast_{\Z/2\Z}\Z/6\Z\cong \langle A\rangle\ast_{\langle -I\rangle} \langle B\rangle$, where \[A=\begin{pmatrix}
0 & 1 \\
-1 & 0 
\end{pmatrix}, B=\begin{pmatrix}
0 & -1 \\
1 & 1
\end{pmatrix}.\]
Thus, $A^2=B^3=-I$ and 

\begin{align*}
	AB&=\begin{pmatrix}
0 & 1 \\
-1 & 0 
\end{pmatrix}
\begin{pmatrix}
0 & -1 \\
1 & 1 
\end{pmatrix}=
\begin{pmatrix}
1 & 1 \\
0 & 1 
\end{pmatrix}=S\\

	AB^2&=\begin{pmatrix}
0 & 1 \\
-1 & 0 
\end{pmatrix}
\begin{pmatrix}
-1 & -1 \\
1 & 0 
\end{pmatrix}=
\begin{pmatrix}
1 & 0 \\
1 & 1 
\end{pmatrix}=T
\end{align*}
Substituting these above, we obtain
\begin{align*}
	\tau_1&=\left(0,(AB)(AB^2)^3(AB)^{-1}\right)=\left(0,ABAB^2AB^2ABA\right)\\
	\tau_2&=\left(x+y,(AB)^2(AB)^3(AB^2)^{-1}\right)=\left(x+y,AB^2AB^2A\right)\\

\tau_3&=\left(x,(AB)^6\right)
\end{align*}

Observe that 
\[\tau_2^{-3}=\left(x+y, AB(AB^2AB^2A)BA\right)=\left(x+y, I\right)\cdot \tau_1,\]
and \[\tau_2\tau_3=\left(x+y, AB^2AB^2A\right)\cdot\left(x, (AB)^6\right)=\left(x, (BA)^3B\right).\]

It follows that $\langle \tau_1,\tau_2,\tau_3\rangle=\langle g_1,g_2,g_3\rangle,$ where 
\begin{align*}g_1&=(x+y,I)=\tau_1\tau_2^3,\\g_2&=\Big(0,(AB^2)^2A\Big)=g_1\tau_2=\tau_1\tau_2^4\\ g_3&=\left(x, (BA)^3B\right)=\tau_2\tau_3.
\end{align*}

The images of $(AB^2)^2A,(BA)^3B\in\SL_2(\Z)$ under the projection to $\PSL_2(\Z)$ generate a free group of rank $2$. Indeed, each has infinite order, and since they start and end in different letters,  their inverses do too. Thus no cancellation occurs after taking products. It follows that $\langle g_2,g_3\rangle\cong F_2$. Now observe that $g_1$ commutes with both $g_2$ and $g_3$; hence, the image of the monodromy is isomorphic to $\Z/2\Z\times F_2.$ The extension of the fundamental group
$$1\to F_3\to \pi_1(\Bl_O(\PP^2)\setminus (C\cup_{i=1}^3 L_i))\to F_3\to 1,$$ with infinite monodromy $\mathbb Z/2\mathbb Z\times F_2$ and infinite kernel of the monodromy. Clearly, $F_3$ has the LIP Property and trivial center, so by Theorem \ref{thm:LIP-Obstruction}, we have the following theorem:

\begin{theorem}\label{InfiniteMono}
The fundamental group of the complement of the union of 3-cuspidal quartic and the singular fibers, $\pi_1(\Bl_O(\PP^2)\setminus (C\cup L))$, is not CAT(0).
\end{theorem}

When we fill in all the singular fibers, $\pi_1(\Bl_O(\PP^2)\setminus C)$ is finite and hence CAT(0), while here $\pi_1(\Bl_O(\PP^2)\setminus (C\cup L)$ is not. A natural question is whether adding back a subset of the singular fibers makes a difference.

Let $L_0$ be the union of the two fibers containing cusps. Since the fibers coming from tangencies depend on the projection, a reasonable modification is to add back the fibers coming from tangencies, i.e., setting $\tau_2=1$ and $\tau_1\tau_2\tau_3=1$. The presentation of the fundamental group becomes
$\pi_1(\Bl_O(\PP^2)\setminus L_0)=\langle S|R\rangle$, where
\begin{align*}
S&=\{\alpha_1,\alpha_2,\beta_2,\tau_1\}\\

  R &= \left\{ \begin{array}{l}
   \alpha_1^{\tau_1}=\alpha_1, \alpha_1=(\overline\beta_2\overline\alpha_2\overline\alpha_1)^{\alpha_1\beta_2},\alpha_1^{\overline\tau_1}=\alpha_2^{\alpha_1\alpha_2\alpha_1}; \\
   
   \alpha_2^{\tau_1}=\beta_2^{\alpha_2\beta_2\alpha_2}, \alpha_2=\alpha_2^{\alpha_1\overline\alpha_2\overline\alpha_1\overline{\beta_2}},\alpha_2^{\overline\tau_1}=\alpha_1^{\alpha_1\alpha_2};\\
   
\beta_2^{\tau_1}=\alpha_2^{\alpha_2\beta_2},\beta_2^{\overline\tau_1}=(\overline\beta_2\overline\alpha_2\overline\alpha_1)^{\overline\alpha_2\overline\alpha_1\beta_2}.
  \end{array}\right\}

\end{align*}
Using that $\alpha_1^{\tau_1}=\alpha_1$ and $\alpha_1^{\overline\tau_1}=\alpha_2^{\alpha_1\alpha_2\alpha_1}$, it follows that 
$\alpha_1\alpha_2\alpha_1=\alpha_2\alpha_1\alpha_2,$ and by
$\alpha_1=(\overline\beta_2\overline\alpha_2\overline\alpha_1)^{\alpha_1\beta_2}$, we obtain $$ \alpha_2=\overline{\alpha_1}\overline{\beta_2}\overline{\alpha_1}.$$ In particular, we can solve for $\beta_2$ in terms of $\alpha_1,\alpha_2$, so $\langle\alpha_1,\alpha_2,\beta_2\rangle=\langle \alpha_1,\alpha_2\rangle$. By $\alpha_2^{\overline\tau_1}=\alpha_1^{\alpha_1\alpha_2}={\alpha_1\alpha_2}\alpha_1\overline\alpha_2\overline\alpha_1$ and $\alpha_1\alpha_2\alpha_1=\alpha_2\alpha_1\alpha_2$, it follows that 
$\alpha_2^{\overline\tau_1}=\alpha_2.$




Therefore $\tau_1$ commutes with $\alpha_1$ and $\alpha_2$. The existence of the projection onto $\langle\tau_1\rangle\cong \Z$ then implies
$$\pi_1(\Bl_O(\PP^2)\setminus (C\cup L_0))\cong \langle \tau_1\rangle \times \langle \alpha_1,\alpha_2\rangle\cong \Z\times \pi_1(\Bl_O(\PP^2)\setminus C)$$
where $\langle \alpha_1,\alpha_2\rangle\cong \pi_1(\Bl_O(\PP^2)\setminus C)$ since killing $\tau_1$ is equivalent to filling all the singular fibers. In particular, since $\pi_1(\Bl_O(\PP^2)\setminus C)$ is finite,  $\pi_1(\Bl_O(\PP^2)\setminus (C\cup L_0))$ is CAT(0) by Lemma \ref{lem:VAbelian-Is-CAT(0)}. Thus, we have proved

\begin{proposition}
    Let $C$ be the 3-cuspidal quartic, and let $L_0$ be the union of the two singular fibers containing cusps.  Then $\pi_1(\Bl_O(\PP^2)\setminus (C\cup L_0))$ is CAT(0).
\end{proposition}
In contrast, given any plane curve complement with infinite monodromy, deleting nonsingular fibers in addition to the singular fibers will make the resulting complement, not CAT(0).  Indeed, deleting a nonsingular fiber creates an additional generator in the fundamental group of the base that acts trivially on the fiber.  In particular, the new monodromy will have the same monodromy but an infinite kernel. For some examples of curves with infinite monodromy, we reference several results obtained by Moishezon:

\begin{theorem}[Theorems 1, 2, and 3, \cite{Moishezon81}]\label{thm:Moish123}~
\begin{enumerate}[(i)]
\item Let  $C$ be an irreducible plane curve of degree $d$ that is either smooth or rational with nodes as the only singularities. Then, for a generic projection, the image of braid monodromy is the full Artin braid group $B_d$. 
\item Let $C$ be a union of $d$ lines in general position.  Then, for a generic projection, the image of braid monodromy is the pure braid group $P_d$.
\end{enumerate}
\end{theorem}

As a quick application of Theorem \ref{thm:Moish123}, we obtain the following corollary, giving infinitely many examples of non-CAT(0) plane curve complements. Similar considerations apply to case (ii) above.  
\begin{corollary}
Let $C$ be an irreducible plane curve of degree $d\geq 4$ that is smooth or rational with nodes as the only singularities, and let $L=\cup_{i=1}^nL_i$ be the union of singular fibers for a generic projection. Then $\pi_1(\Bl_O(\PP^2)\setminus (C\cup L))$ is not CAT(0).
\end{corollary}
\begin{proof}
    For $d\geq 4$, $\Mod(S_{0,d})$ is infinite, and by Theorem \ref{thm:Moish123}(i), in this case the braid monodromy is all of $B_d$.  Since both compositions $B_d\rightarrow B_{d}(S^2)$ and $B_d(S^2)\rightarrow \Mod(S_{0,d})$ are surjective, it follows that the geometric monodromy $\rho\colon \pi_1(S_{0,n})\rightarrow \Mod(S_{0,d})$ is surjective. However, as $\Mod(S_{0,d})$ is not free, $\rho$ cannot be injective. We conclude that $\pi_1(\Bl_O(\PP^2)\setminus (C\cup L))$ is not CAT(0) by Theorem \ref{thm:LIP-Obstruction}.
\end{proof}

\subsection{Semi-universal deformations of ADE singularities}\label{UDAEDS}
\text{}\\
Recall  (cf.\cite{greuel2007introduction}, Chapter II, Cor.1.17),  that for a germ $X$ given by $f(x_1,\cdots, x_n)=0$ with an isolated singularity at the origin, 
the semi-universal deformation $\widetilde X \rightarrow \C^\tau$  of $f$ can be 
given as the map induced by projection ${\widetilde X} \subset {{\mathbb C^n}\times {\mathbb C}^{\tau}} \rightarrow {\mathbb C}^{\tau}$ 
where $\widetilde X$ is the hypersurface $F(x,t)= f(x)+\sum_1^{\tau} t_jg_j(x)$ with $g_j(x)$ being 
a $\mathbb C$-basis of the Tjurina algebra ${\mathbb C}\{x_1,\cdots, x_n\}/\langle f,{{\partial f} \over {\partial x_1}},\cdots,{{\partial f} \over {\partial x_n}}\rangle$. 
In the case of a weighted homogeneous singularity $f=0$, the Tjurina algebra coincides with the Milnor algebra.  Moreover, the complement to the 
discriminant of $F$ (\emph{i.e.} the subset of $\mathbb C^{\tau}$ where $F(x,t)=0$ is smooth, which is also weighted homogeneous if $f$ is)
is a retract of the complement to the corresponding hypersurface in $\mathbb C^{\tau}$. 

We consider here semi-universal deformations of germs corresponding to simple singularities of plane curves, which are indexed by a root system $R$ of type $A_n$, $D_n$, $E_6,$ $ E_7,$ or $ E_8$. Accordingly, we have polynomials $f_R(x)$ and $F_R(x,t)$ as in the preceding discussion. For example, in the $E_6$ case one has a germ $f_{E_6}(x)=x_1^3+x_2^4$, with 
$F_{E_6}(x,t)=x_1^3+x_2^4+t_1+t_2x_1+t_3x_2+t_4x_2^2+t_5x_1x_2+t_6x_1x_2^2$. Let $U_R\subset \C^\tau$ be the complement of the discriminant locus of $F_R$, and let $V_R\subset \C^n\times \C^\tau$ be its pre-image.  The projection $\pi_R\colon V_R\rightarrow U_R$ is a locally trivial fiber bundle with the fiber diffeomorphic to an affine curve $C_R$.


 Topologically, $C_R$ is homeomorphic to the interior of a surface with boundary having first Betti number $\mu(R)$ and $b(R)$ boundary components,  where $\mu(R)$ is the Milnor number and $b(R)$ is the number of branches of singularity corresponding to $R$. The germs of functions having simple singularities can be 
chosen to be weighted homogeneous, \emph{i.e.} generic linear combinations of monomials: $x_1^{i_1}x_2^{i_2}$ where ${i_1\over w_1}+{i_2 \over w_2}=1$ for appropriate weights $(w_1,w_2)$, in which case $\mu(R)=(w_1-1)(w_2-1)$ (cf. \cite{dimca2012singularities}).
The number of branches $b(R)$ coincides with the number of irreducible factors of the germ or, alternatively, the number of connected components of the strict transform of the germ in a resolution of its singularity. For example, $\mu(E_6)=(3-1)(4-1)=6$, and  $b(E_6)=1$.

Setting $g(R)=\mu(R)+b(R)-1$, we see that $\pi_1(C_R)\cong F_{g(R)}$, the free group on $g(R)$ generators. On the other hand, the base space $U_R$ is biholomorphic to the quotient of 
 the finite affine hyperplane complement corresponding to $R$ by the Weyl group $W_R$. According to Deligne's work, this quotient is a classifying space for $A_R$, the Artin group of type $R$. Since $C_R$ is also aspherical, this implies the total space $V_R$ is aspherical as well. In particular,  $\pi_2(U_R)=1$ and one has a short 
 exact sequence $$1 \rightarrow F_{g(R)} \rightarrow \pi_1(V_R) \xrightarrow{(\pi_R)_*} A_R\rightarrow 1$$
 
 The locally trivial fiber bundle $\pi_R\colon V_R\rightarrow U_R$ induces a geometric monodromy homomorphism $\phi_R\colon A_R\rightarrow \Mod(C_R)$, where $\Mod(C_R)$ is the mapping class group of $C_R$ fixing the boundary pointwise.  It was shown by Perron-Vannier \cite{PerronVannier} that 
$\phi_R$ is injective when $R$ has type $A_n$ or $D_n$. In contrast,  Wajnryb showed that $\ker\phi_R$ is infinite when $R=E_6$, $E_7$, or $E_8$. In all cases, the standard Artin generators map to Dehn twists on the surface, hence the image of the geometric monodromy is necessarily infinite. We will leverage these to facts to show that $\pi_1(V_R)$ cannot be CAT(0) in these cases.

\subsection{Wajnryb's theorem and non-CAT(0) families}
\begin{definition}Let $A$ be an Artin group. A representation $\phi\colon A\rightarrow\Mod(S)$ is called \emph{geometric} if each of the standard Artin generators is sent to Dehn twists about simple closed curves on $S$.
\end{definition}
\begin{theorem}[Theorem 3, \cite{wajnryb1999artin}]
    Let $R$ be the Artin group associated to either $E_6$, $E_7$, or $E_8$. There does not exist an injective geometric representation $\phi\colon A\rightarrow \Mod(S)$ for any $S$.
\end{theorem}

Let $\gamma_1,\gamma_2$ be essential simple closed curves on a finite type surface, and let $T_{\gamma_i}$ be the Dehn twist along $\gamma_i$, $i=1,2$. We may assume that $\gamma_1$ and $\gamma_2$ are in general position and have been isotoped to have a minimal number of intersections. The subgroup generated by $\{T_{\gamma_1},T_{\gamma_2}\}$ is determined as follows \[\langle T_{\gamma_1},T_{\gamma_2}\rangle\cong\left\{\begin{array}{cc}\Z^2,&|\gamma_1\cap \gamma_2|=0\\B_3,&|\gamma_1\cap \gamma_2|=1\\F_2,&|\gamma_1\cap \gamma_2|\geq 2\end{array}\right.\]where $B_3$ is the 3-strand braid group of the disk. Under a geometric representation, each of the standard Artin generators for the $E_6$, $E_7$ or $E_8$ type must go to a non-separating simple closed curve of $S$. A regular neighborhood of the collection of curves giving the representation for $E_6$ must look as in Figure \ref{fig:surface}. 

\begin{figure}[h]
    \centering
\begin{tikzpicture}

\draw (4,4) to[out=180,in=0] (-4,4);
\draw (4,0) to[out=180,in=0] (-4,0);
\draw (-4,4) to[out=180,in=90] (-5.5,2);
\draw (-5.5,2) to[out=-90,in=-180] (-4,0);

\draw (-3.7,1.8) to[out=80,in=100] (-3.2,1.8);
\draw (-3.9,2) to[out=-80,in=-80] (-3,2);
\draw [blue](-2.7,2) to[out=90,in=90] (-4.3,2);
\draw [blue](-4.3,2) to[out=-90,in=-90] (-2.7,2);

\draw (-1.2,1.8) to[out=80,in=100] (-0.7,1.8);
\draw (-1.4,2) to[out=-80,in=-80] (-0.5,2);

\draw [blue](-0.2,2) to[out=90,in=90] (-1.8,2);
\draw [blue](-1.8,2) to[out=-90,in=-90] (-0.2,2);

\draw (1.3,1.8) to[out=80,in=100] (1.8,1.8);
\draw (1.1,2) to[out=-80,in=-80] (2,2);

\draw [blue](2.3,2) to[out=90,in=90] (0.7,2);
\draw [blue](0.7,2) to[out=-90,in=-90] (2.3,2);

\draw [blue](-3.1,1.8) to[out=90,in=90] (-1.25,1.8);
\draw [blue][dotted](-3.1,1.8) to[out=-90,in=-90] (-1.25,1.8);

\draw [blue](-0.6,1.8) to[out=90,in=90] (1.25,1.8);
\draw [blue][dotted](-0.6,1.8) to[out=-90,in=-90] (1.25,1.8);

\draw [blue](-0.95,1.95) to[out=170,in=-170] (-0.85,4);
\draw [blue][dotted](-0.95,1.95) to[out=10,in=-10] (-0.85,4);

\draw (-0.9,3) node  {$a_1$};
\draw (-1,2.2) node  {$a_4$};
\draw (-3.5,2.2) node  {$a_2$};
\draw (-2.2,1.7) node  {$a_3$};
\draw (0.2,1.7) node  {$a_5$};
\draw (1.6,2.2) node  {$a_6$};

\draw (4,2) ellipse (0.5 and 2);

\filldraw  (6,2) circle (2pt);
\filldraw  (7,2) circle (2pt);
\filldraw  (8,2) circle (2pt);
\filldraw  (9,2) circle (2pt);
\filldraw  (10,2) circle (2pt);
\filldraw  (8,3) circle (2pt);
\draw (6,2) -- (7,2);
\draw (7,2) -- (8,2);
\draw (8,2) -- (9,2);
\draw (9,2) -- (10,2);
\draw (8,3) -- (8,2);

\node at (8, 3.3) {$a_1$};
\node at (6,1.7) {$a_2$};
\node at (7,1.7) {$a_3$};
\node at (8,1.7) {$a_4$};
\node at (9,1.7) {$a_5$};
\node at (10,1.7) {$a_6$};

\end{tikzpicture}
\caption{The Dynkin diagram of $E_6$ and the surface with boundary $C_{E_6}$ associated with it, along with the curves which define the geometric representation of $A_{E_6}$. }
\label{fig:surface}

\end{figure}
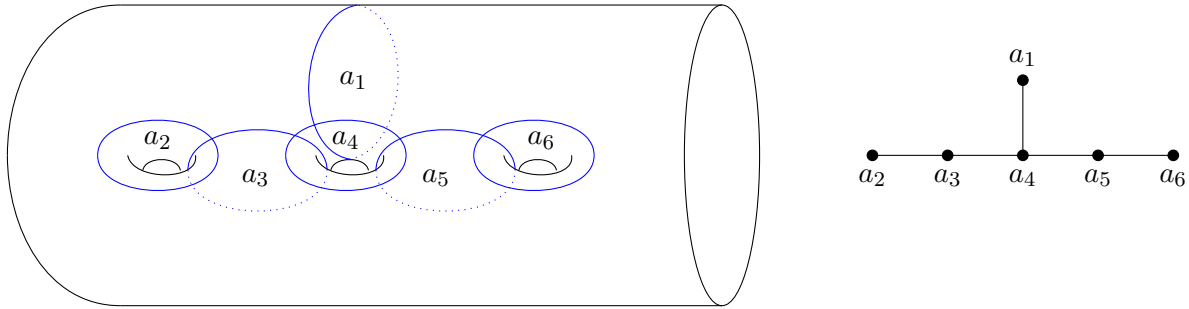

Wajnryb uses the natural Garside structure on these groups to show that any such representation cannot be injective by finding an explicit element in the kernel of $\phi$. Specifically, let $\{a_1,\ldots, a_6\}$ be the standard Artin generators of $E_6$ as in Figure \ref{fig:surface}.  The Garside element is $\Delta=(a_2a_4a_6a_1a_3a_5)^6$. Wajnryb finds a factorization $\Delta=bc$ and shows that \begin{equation}\label{eqn:Wajnryb-Relation}
    \phi(ba_1ca_1ba_1)=\phi(a_1ba_1ca_1b)
\end{equation} but that $ba_1ca_1ba_1\neq a_1ba_1ca_1b$ in $A_{E_6}$. By a result of van der Lek \cite{vanderLek} (see also \cite{Paris}), for any root system $R$ containing $E_6$ there is an injection $A_{E_6}$ into $A_R$. Wajnryb then argues that for any geometric representation $\phi$ of $A_R$, Equation \ref{eqn:Wajnryb-Relation} must still be satisfied, hence $\phi$ cannot be injective either.  This takes care of the cases $R=E_7$ and $R=E_8$. The following appears as Theorem \ref{thm:mainB} in the introduction.

\begin{theorem}\label{Rootsystemthem} Let $R$ be the root system of type $E_6$, $E_7$, or $E_8$ and let $V_R$ be the universal family associated to the singularity of type $R$. Then $\pi_1(V_R)$ is not CAT(0).
\end{theorem}
\begin{proof}
    Let $G=\pi_1(V_R)$. The center $Z=Z(A_R)$ of $A_R$ is infinite cyclic and virtually splits off as a direct factor. That is, there exists a finite index subgroup $A'\leq A_R$ such that $Z\leq A'$ and $A'\cong P'\times Z$.  Explicitly, the abelianization map $\alpha\colon A_R\to\Z$ is obtained by sending every positive Artin generator to $1\in \Z$. Since the center is generated by a positive word, its image under $\alpha$ is a nontrivial positive integer $n\in \Z$. Thus, we may take $A'=\alpha^{-1}(n\Z)$, which has finite index in $A_R$, and contains $Z$ as a retract. In particular, we can take $P'=\ker(\alpha)$ to obtain $A'=P'\times Z$.

    If $V'$ is the cover of $V_R$ corresponding to $G'=(\pi_R)_*^{-1}(A')$ we have an extension
    \[1\to F_{g(R)}\to G'\xrightarrow{(\pi_R)_*} A'\to 1\]
    Since $G'\leq G$ has finite index, if $G$ is CAT(0) then $G'$ is too. Under the monodromy action $\phi$, the center of $A_R$ acts by a multi twist along the boundary of $C_R$.  In particular, since $Z(A_R)$ is free, there is a splitting $\sigma\colon Z\rightarrow G'$ such that $\sigma(Z)$ is central in $G'$. Since $Z$ is clearly still a retract of $G'$, we obtain a direct product decomposition $G'=H'\times \sigma(Z)$.  By a result of Bowers--Ruane \cite{BowersRuane}, $G'$ is CAT(0) if and only if $H'$ is. 
    
    By construction, $H'$ fits into a short exact sequence \[1\to F_{g(R)}\to H'\to P'\to 1\]
    We claim that this extension satisfies the hypotheses of Theorem \ref{thm:LIP-Obstruction} and that the restriction of the monodromy $\phi$ to $P'$ contains the nontrivial infinite kernel found by Wajnryb.  This will imply $H'$ is not CAT(0), and thus that $G$ is not CAT(0) either. First, $F_{g(R)}$ has trivial center as $g(R)\geq 2$.  Secondly, as $P'\cap Z=1$, $P'$ maps isomorphically onto a finite index subgroup of $A_R/Z$.  The latter is acylindrically hyperbolic by a result of Calvez--Wiest \cite{CalvezWiest}. Thus, $P'$ is itself acylindrically hyperbolic and has the LIP property by Proposition \ref{prop:AC-is-LIP}. 
    
    Finally, we verify that the restriction of $\phi$ to $P'$ has infinite kernel. Since $F_{g(R)}$ is torsion-free, any torsion element of $H'$ injects into $P'\leq A_R/Z$. On the other hand, by \cite[Theorem 4.5]{Bestvina-Artin}, any torsion element of $A_R/Z$ is conjugate to the image of a positive word. Since $P'=\ker\alpha$, this means it is torsion-free. As $A_R$ is torsion-free, it suffices to show that $P'$ contains an element of $\ker\phi$. Recall from Equation (\ref{eqn:Wajnryb-Relation}) and the discussion above that $A_R$ contains elements $a_1,b,c$  satisfying  \[\phi(ba_1ca_1ba_1)=\phi(a_1ba_1ca_1b) \quad\text{ and } \quad ba_1ca_1ba_1\neq a_1ba_1ca_1b.\] In particular, this means $g=(ba_1ca_1ba_1)(a_1ba_1ca_1b)^{-1}\in \ker\phi$. On the other hand, $g$ is clearly in the commutator subgroup of $A_R$, which is just $\ker\alpha=P'$.
\end{proof}

\appendix
\section{Background from the algebraic geometry of plane curves}
\subsection{Fundamental groups of complements.}\label{A1}
We give a brief overview of the topology of plane curves for non-experts, including some classical results on the numerical invariants.  We also describe a construction that associates with each plane curve, a locally trivial fibration, such that the fundamental group of the total space is an extension of a free group by another free group. 

\subsubsection{Pencils of lines.}\label{pencils} 
Let $C\subset \PP^2$ be a plane curve and let $O=[a:b:c]$ be a point in $\PP^2\setminus C$.
The lines in $\PP^2$ containing $O$ are parametrized by a projective line $B=Z(ax+by+cz)$ in the Grassmannian of lines in $\PP^2$, which is itself isomorphic to $\PP^2$ and classically called {\it the dual plane} ${\PP^2}^{\vee}$. Since each line $L=Z(l_1x+l_2y+l_3z)$ passing through $O$ satisfies $l_1a+l_2b+l_3c=0$, we can identify the line $L$ with the point $[l_1:l_2:l_3]\in B$. To $O$ one associates a map $p: \PP^2\setminus \{O\}\rightarrow B$ (called {\it the projection with center at $O$}) that assigns to $P \in \PP^2\setminus  \{O\}$ the unique line in $\PP^2$ containing $O$ and $P$. If one identifies (non-canonically) $B$ with a line in $\PP^2\setminus  \{O\}$ then this map assigns to $Q$ the unique intersection point of the line $OP$ with $B$. The projection of the complement to $O$ is a holomorphic map that cannot be extended to a map of $\PP^2$ but can be extended to the map of the Hirzebruch surface $F_1$ that is the blow-up $\Bl_O(\PP^2)$ of $\PP^2$ at $O$. The fibers of this extension are projective lines, while the fibers of $p$ are affine lines. 

For $b\in B$, let $L_b$ be the line passing through $O$ and $b$. The restriction of this projection to $C$ gives a (ramified) covering $p_C: C\rightarrow B$ such that for generic $b \in B$ the corresponding line $L_b$ 
intersects $C$ at smooth points only and transversally. In this case, the number of preimages is equal to the degree $d$ of $C$. 
For arbitrary $b\in B$,  the sum of the {\it intersection multiplicities} of $L_b$ with $C$  is equal to  $d$. The intersection multiplicity is greater than one if either the line $L_b$ contains a singular point of $C$ or it is tangent to $C$. If the order of tangency of $L_b$
is greater than 2, it is called a {\it flex} (if the order of tangency is greater than $3$, a {\it higher-order flex}), and if $L_b$ contains more than one point of tangency, it is called a {\it multiple tangent} (a {\it bitangent} if this number is $2$). 

Information about the special fibers of pencils of lines is contained 
in the classical Pl\"ucker formulas (cf. \cite{kleiman1977enumerative}). For example, if $C$ contains as singularities only $\delta$ ordinary nodes  and $\kappa$ ordinary cusps 
(\emph{i.e.} the singularities having local equation $y^2=x^2$ and $x^2=y^3$, respectively) and the pencil contains only the lines that either have a tangency
of order $2$ and are transverse to the lines of the tangent cones of all singular points (such pencils are called {\it generic}) then the number $d^{\vee}$ of the lines in the pencil containing 
fewer than $d$ intersections with $C$ is 
\begin{equation}\label{plucker}
    d^{\vee}=d(d-1)-3\kappa -2\delta
\end{equation}
It is convenient to describe the singular fibers of non-generic pencils in terms of the dual curve of $C$, \emph{i.e.} the curve $C^{\vee}$ in ${\PP^2}^\vee$ that is the projective closure 
of the set of points formed by the lines tangent to $C$ at its smooth points. Singularities of $C^{\vee}$ correspond to the flexes and multiple tangents of $C$ (cf. \cite{kleiman1977enumerative})
while the ordinary flexes (resp. ordinary bitangents) of $C$ correspond to the ordinary cusps (resp. ordinary nodes) of $C^{\vee}$. 
Thus the pencil in $\PP^2$ consisting of the lines containing $O$ corresponds to a line $P^{\vee}$ in 
${\PP^2}^{\vee}$ whose special fibers of the pencil corresponding to the intersection points of $P^{\vee}$ with $C^{\vee}$ or a line in ${\PP^2}^{\vee}$ 
corresponding to the pencils in $\PP^2$ centered at the singular points of $C$. 

The numerical data of non-generic pencils can be obtained 
using the remaining Pl\"ucker formulas for curves with ordinary nodes cusps, flexes, and bitangents (the number of the latter two denoted by $\iota$ and 
$\tau$, respectively) are given by
\begin{align}\label{plucker2} 
 d&=d^{\vee}(d^{\vee}-1)-3\iota-2\tau, \\\label{plucker3}
 \iota&=3d(d-2)-6\delta-8\kappa, \\ \label{plucker4}\kappa&=3d^{\vee}(d^{\vee}-1)-6\tau-8\iota.
\end{align} 
The symmetry of the 4 equations in (\ref{plucker})--(\ref{plucker4}) reflects the duality relation ${C^{\vee}}^{\vee}=C$ (cf. \cite{kleiman1977enumerative}).
\begin{example}\label{ex:3cuspidalquartic}
Let $C$ be a cubic curve with one node (\emph{e.g.} $y^2z=x^3-x^2z$) and thus numerical invariants $\delta=1,\kappa=0$. Then $d^{\vee}=4, \iota=3,\tau=0$.
Hence, $C^{\vee}$ is a 3-cuspidal quartic. It has no nodes and, by duality, has one bitangent corresponding to the node of $C$. There are 3 lines in $\PP^2$ connecting the node of $C$ with the intersection point of one of 3 pairs of inflectional tangents of $C$. 
A pencil in ${\PP^2}^{\vee}$ corresponding to either of these 3 lines $L_i$ in $\PP^2$, has 4 special fibers containing respectively two tangency points 
(this line in ${\PP^2}^{\vee}$ corresponds to the node of $C$), two cusps (the line corresponding to the intersection point in $\PP^2$ of two inflectional lines), one cusp (this fiber corresponds to the intersection point of $L_i$ with remaining inflectional tangent) and finally, the fiber corresponding to the intersection point of $L_i$ 
with $C$  outside of the node and containing one simple tangency point.
\end{example} 

\subsubsection{Braid monodromy} Let $S \subset B$ be the set of singular values, and let $L=\cup_{s\in S}L_s$ be the set of singular lines. Then $p_C: C \setminus L \rightarrow B\setminus S$ is an unramified covering space, and the restriction 
of $p\colon \PP^2\setminus (C\cup L)\rightarrow B\setminus S$ (also denoted $p$) is a locally trivial fibration. Both its base and fibers are complements to a finite set of points in $\PP^1$:
this is the set $S\subset B$ in the former case and $O\cup p_C^{-1}(b)\subset L_b,~b\in B\setminus S$ in the latter. Fix a basepoint $b_0\in B\setminus S$, a point $x_0\in X_0=p^{-1}(b_0)$ and let $i_0\colon X_0\rightarrow \PP^2\setminus (C\cup L)$ be the inclusion.  
Since $S\neq \emptyset$, $B\setminus S$ is aspherical,  hence the long exact sequence of the fibration determines an extension
\begin{equation}\label{extension}
   1 \rightarrow \pi_1(X_0,x_0) \xrightarrow[]{(i_0) _* }\pi_1(\PP^2\setminus (C\cup L),x_0) \xrightarrow[]{p_* } \pi_1(B\setminus S,b_0) \rightarrow 1
\end{equation}
The identification of $\pi_1(B\setminus S,b_0)$ with the free group can be made by a choice of simple closed curves $\{\gamma_s\mid s\in S\}$ in $B\setminus S$ based at $b_0$ but otherwise disjoint, and which circle each puncture in $S$. Similarly, a choice $\{\delta_i\mid 1\leq i\leq d\}$ of loops in $X_0$ based at $x_0$ give an identification $\pi_1(F_0,f_0)$ with the free group $F_d$. 

The extension can be described as follows. Let $\gamma\colon[0,1]\rightarrow B\setminus S$ represent a loop based at $b_0 $ (\emph{i.e.} satisfying $\gamma(0)=\gamma(1)=b_0$). Pulling back $p\colon \PP^2\setminus (C\cup L)\rightarrow B\setminus S$ via $\gamma$ yields a punctured $\PP^1$-bundle over the interval $[0,1]$ which is trivial since $[0,1]$ is contractible. A trivialization  $\Phi\colon \gamma^*(\PP^2\setminus (C\cup L))\rightarrow [0,1]\times X_0$ such that $\Phi_0\colon X_0\rightarrow X_0$ is the identity determines an orientation-preserving diffeomorphism $\Phi_1\colon X_0\rightarrow X_0$. We can assume this diffeomorphism is the identity outside of a sufficiently large disk containing $p_C^{-1}(b_0)$ (we may choose $x_0$ sufficiently close to $O$ so that it also lies outside this disk). The isotopy class of this diffeomorphism is independent of homotopy class of $\gamma$ rel endpoints, as well as the choice of trivialization. The group of isotopy classes of diffeomorphisms of the disk with $d$ punctures is the braid group $B_d$.

\begin{definition}
The homomorphism $\beta\colon \pi_1(B\setminus S,b_0) \rightarrow B_{d}$ into the braid group 
assigning the braid $\beta(\gamma) \in B_{d}$ corresponding to the class of isotopy of the above diffeomorphism is called {\it braid monodromy} associated to $C$.
\end{definition} 

\begin{example} {\it Braids corresponding to singular fibers of the pencil.} Each fiber $p^{-1}(s), s\in S$ contains either a singular point of $C$ or a tangency point. 
{Let $\gamma_s$ be a path in $B$ connecting $b_0$ and $s$ (avoiding other points in $S$). Then $p^{-1}(s)\cap C$ contains fewer points than  $p^{-1}(b_0)$ }and one can show that 
there is a 1-complex (vanishing fold) such that extension of the diffeomorphism of pairs $I \times (\pi^{-1}(b_0),C\cap p^{-1}(b_0) \rightarrow (\PP^2,C)$ well defined over $I$ with removed preimage of 1, to the map of $1 \times (\pi^{-1}(b_0),C\cap p^{-1}(b_0)$ contracts this complex to a point.

If $s$ corresponds to a point with multiplicity $2$ or a simple tangency point, then this vanishing fold is just a segment that gets contracted to a point as one moves along $\gamma$ from $b_0$ to $s$. 
In this case, it is called {\it the vanishing cycle}. Denote it $\delta_s$ and let $\gamma(s)$ be a loop that is union of $\gamma_s \setminus D_s\cap \gamma_s$ where $D_s$ is a small disk centered at $s$, the boundary of $D_s$ 
oriented counterclockwise, and then returning back to $b_0$ along $\gamma_s$. The braid monodromy assigns to $\gamma(s)$ some power of the Dehn twist about $\delta_s$.
More precisely, if the projection looks locally as a projection of $y^k=x^2$ onto $x$-axis, the braid corresponding to $\gamma_s$ is the $k$-th power of the Dehn twist, and the braid is conjugate to $a_1^k$ where $a_1$ is one the standard generators of the braid group (cf.\cite{Moishezon81}, \cite{cogolludo2011braid}). 
\end{example}

\begin{theorem}[Zariski--van Kampen]\label{app:Zariskitheorem}
Let $C \subset \PP^2$ be an algebraic curve, $P \in \PP^2 \setminus C$ and with the same notation as above, let $\beta\colon\pi_1(B \setminus S,b_0) \rightarrow B_d$ be the braid monodromy. Then:
\begin{enumerate}[(a)]

\item The split extension (\ref{extension}) is determined by $\varphi \colon \pi_1(B\setminus S,b_0) \rightarrow \Aut(F_d)$, where $\varphi$ is the composition of $\beta$ and the Artin representation $\rho_A\colon B_d\rightarrow \Aut(F_n)$ (cf. \cite{farb2011primer})

\item Choose $\tilde b \in B, \tilde b \ne b_0$. Then the group $\pi_1(\PP^2 \setminus p^{-1}(S \cup \tilde b) \cup C)$ has the following presentation ($N=|S|$)
\begin{equation}
 \pi_1(\PP^2 \setminus p^{-1}(S \cup \tilde b) \cup C)=\langle t_1,\ldots, t_N, x_1,\ldots, x_d \mid x_i^{t_j}=\beta_j(x_i)\rangle 
\end{equation} 
(here $j=1,\cdots, N, i=1, \cdots d$)

\item (complements to projective and affine curves) If $s \in B\setminus S$ then $\pi_1(\PP^2\setminus C)$ (resp.  $\pi_1(\mathbb C^2\setminus C)$ where $\mathbb C^2=\PP^2\setminus  p^{-1}(\tilde b))$) has the following presentation:
\begin{equation}\label{projgroup} \pi_1(\PP^2\setminus C)=\{ x_1,\ldots, x_d \vert \beta_j(x_i)=x_j, x_1\cdots x_d=1\}
\end{equation}
resp.
\begin{equation} \pi_1(\C^2\setminus C)=\{ x_1,\ldots, x_d \vert \beta_j(x_i)=x_j\}
\end{equation}
\end{enumerate}
\end{theorem}

The fundamental group $\pi_1(F_1\setminus C)$ where $F_1$ is the blow-up of $\PP^2$ at $P$ has expression similar to (\ref{projgroup}) 
The only change was the replacement of Artin's braid group with the braid group of spheres. Blow-up does not change, of course, the fundamental group.

\subsection{Generic projections}\label{A2}

Let $V$ be a smooth surface in $\PP^N$. Let $H\cong\PP^{N-3} \subset \PP^N\setminus V$
be a linear subspace. The Schubert subvariety $B$ of the Grassmannian of linear $N-2$-subspaces in $\PP^N$ 
containing $H$ can be canonically identified with $\PP^2$. Extending the construction in \ref{pencils}, 
one can consider the map $p: \PP^N\setminus H \rightarrow B$ assigning to a point $P$ the linear subspace of $\PP^N$ spanned
by $P$ and $H$ and its restriction $p_V: V \rightarrow B$ to $V$. For a Zariski open subset in $B$, the corresponding linear subspaces are transverse to $V$ and intersect $V$ at $d=\deg (V)$ points. The complement is called the branching curve $C_V$ of $p_V$. 

For $H$ in a Zariski open set in the Grassmannian of linear subspaces $\PP^{N-3}\subset \PP^N$, the branching curve $C_V$ is reduced, and singularities of $C_V$ are ordinary cusps and nodes. Projections corresponding to any subspace $H$ in this set are called {\it generic}. The singularities of $C_V$ can be described in terms of the ramification curve $R$ (\emph{i.e.} the curve $R\subset V$ formed by critical points of $p_V$), as the singularities of  
the image of projection ${p_V}\vert_R: R \rightarrow B$. Note that  ${p_V}\vert_R$ is almost never generic as a map of curves (since generic projections of curves have nodes as the only singularities, while cusps of $C_V$ are unavoidable for almost all $V$). 

Branching curves of generic projections form an important class of plane singular curves, and their fundamental groups have been extensively studied. For past results on fundamental groups of the complements to the branching curves to generic projections, as well as more recent calculations, see \cite{friedman2012fundamental} and the references therein.

\subsubsection{Duals of rational nodal curves.}
Let $C_d \subset \PP^2$ be a rational curve of degree $d$ having only nodes as singularities. Such a curve can be obtained as a generic projection of the smooth rational normal curve in $\PP^d$, $\emph{i.e.}$ the image of $\PP^1$ under the embedding $x_i=s^it^{d-i}$, for $[s:t]\in \PP^1$. It has $d-1 \choose 2$-nodes.
It follows from (\ref{plucker}) and (\ref{plucker2}) that the dual curve $C_d^{\vee}$ has degree $2d-2$, $3d-6$ cusps, and $2(d-2)(d-3)$ nodes. 

\begin{theorem}\cite{zariski1936poincare}
The fundamental group $\pi_1(\PP^2\setminus C_d^{\vee})$ is isomorphic to $B_d(S^2)$, the braid group of sphere on $d$ strands.  
\end{theorem}

Note that $B_d(S^2)\cong B_d/\llangle\sigma_1\cdots\sigma_{d-2}\sigma_{d-1}^2 \sigma_{d-2}\cdots \sigma_1\rrangle$. When $d=3$, $C_3^{\vee}$ is a quartic with 3 cusps and no nodes, and $\pi_1(\PP^2\setminus C_3^\vee)\cong B_3(S^2)$ is a finite group of order 12.

\bibliography{Version1} 
\bibliographystyle{alpha} 
\end{document}